\documentclass[12pt]{article}
	\usepackage{pstricks}
	\usepackage{amsfonts}
	\usepackage{amsmath}
	\usepackage{amssymb}
	\usepackage{mathrsfs}
	\usepackage[colorlinks=true,linkcolor=black,citecolor=black,urlcolor=black]{hyperref}
	\usepackage{enumerate}
	\usepackage{theorem}

	\usepackage{young}

\usepackage[small,heads=LaTeX,nohug]{diagrams}

\psset{unit=1pt}
\psset{arrowsize=4pt 1}
\psset{linewidth=.5pt}

\countdef\x=23
\countdef\y=24
\countdef\z=25
\countdef\t=26

\def\graybox(#1,#2){
\x=#1 \y=#2 
\z=\x \t=\y
\advance\z by 10 
\advance\t by 10 
\psframe[fillstyle=solid,fillcolor=lightgray,linewidth=0pt](\x,\y)(\z,\t) 
}

\def\emptygraybox(#1,#2){
\x=#1 \y=#2 
\z=\x \t=\y
\advance\z by 10 
\advance\t by 10 
\psframe[fillstyle=solid,fillcolor=lightgray,linewidth=0pt,linecolor=lightgray](\x,\y)(\z,\t)}

\def\blankbox(#1,#2){
\x=#1 \y=#2 
\z=\x \t=\y
\advance\z by 10 
\advance\t by 10 
\psframe[linewidth=0pt](\x,\y)(\z,\t)}

    \newcommand{\define}{\textbf}

	\newcommand{\PP}{\mathbb{P}}

	\newcommand{\ZZ}{\mathbb{Z}}	
	\newcommand{\CC}{\mathbb{C}}	
	\newcommand{\EE}{\mathbb{E}}	
	\newcommand{\BB}{\mathbb{B}}	
		
	\newcommand{\B}{\BB}	
	\newcommand{\E}{\EE}	
	\newcommand{\Z}{\ZZ}	
	\newcommand{\C}{\CC}	
	\renewcommand{\P}{\PP}	
	\renewcommand{\O}{\Oo}	
	\newcommand{\Ee}{\mathscr{E}}	
	\newcommand{\Gg}{G}
	\newcommand{\Ff}{\mathscr{F}}	
	\newcommand{\Oo}{\mathscr{O}}	
	\newcommand{\Nn}{N}

	\newcommand{\OOmega}{\mathbf{\Omega}}

	\DeclareMathOperator{\rank}{rank}
	\DeclareMathOperator{\codim}{codim}

	\newcommand{\Gr}{Gr}
	\newcommand{\GGr}{\mathbf{Gr}}
	\newcommand{\Fl}{Fl}

	\newcommand{\isom}{\simeq}

	\newcommand{\pt}{\mathrm{pt}}
	\renewcommand{\phi}{\varphi}
	\renewcommand{\epsilon}{\varepsilon}
	\renewcommand{\tilde}{\widetilde}
    \renewcommand{\setminus}{\smallsetminus}
    
	\newtheorem{theorem}{Theorem}[section]
	\newtheorem{lemma}[theorem]{Lemma}
	\newtheorem{proposition}[theorem]{Proposition}
	\newtheorem{corollary}[theorem]{Corollary}

	{\theorembodyfont{\rmfamily}

	\newtheorem{definition}[theorem]{Definition}
	\newtheorem{remark}[theorem]{Remark}
	\newtheorem{example}[theorem]{Example}
	\newtheorem{exercise}[theorem]{Exercise}
	\newtheorem{claim}{Claim}
	\newtheorem{problem}[theorem]{Problem}
	}

	\newcommand{\qed}{\hfill \mbox{\raggedright \rule{.07in}{.07in}}}

	\newenvironment{proof}
	   {\medskip \noindent \textit{Proof.}~}
	   {\hfill \qed \medskip}

	\newenvironment{proofof}[1]
	   {\medskip \noindent \textit{#1.}~}
	   {\hfill \qed \medskip}

\begin{document}

\title{Introduction to Equivariant Cohomology\\ in Algebraic Geometry}
\author{Dave Anderson\thanks{Partially supported by NSF Grant DMS-0902967. \newline \indent \textit{2010 Mathematics Subject Classification.} Primary 14F43; Secondary 14M15, 14N15, 05E05. \newline \indent \textit{Key words and phrases.} equivariant cohomology, localization, Grassmannian, Schubert variety, Schur function.} }
\date{November 30, 2011}
\maketitle

\begin{abstract}
Introduced by Borel in the late 1950's, equivariant cohomology encodes information about how the topology of a space interacts with a group action.  Quite some time passed before algebraic geometers picked up on these ideas, but in the last twenty years, equivariant techniques have found many applications in enumerative geometry, Gromov-Witten theory, and the study of toric varieties and homogeneous spaces.  In fact, many classical algebro-geometric notions, going back to the degeneracy locus formulas of Giambelli, are naturally statements about certain equivariant cohomology classes.

These lectures survey some of the main features of equivariant cohomology at an introductory level.  The first part is an overview, including basic definitions and examples.  In the second lecture, I discuss one of the most useful aspects of the theory: the possibility of localizing at fixed points without losing information.  The third lecture focuses on Grassmannians, and describes some recent positivity results about their equivariant cohomology rings.
\end{abstract}

\pagebreak

\tableofcontents

\section{Lecture One: Overview}\label{s:lec1}

A general principle of mathematics holds that one should exploit symmetry to simplify a problem whenever possible.  A common manifestation of symmetry is the action of a Lie group $G$ on a topological space $X$ --- and when one is interested in understanding the cohomology ring $H^*X$, the \emph{equivariant cohomology} $H_G^*X$ is a way of exploiting this symmetry.

Topologists have long been interested in a sort of converse problem: given some topological (or cohomological) information about $X$, what can one say about the kinds of group actions $X$ admits?  For example, must there be fixed points?  How many?  It was in this context that Borel originally defined what is now called \emph{equivariant cohomology}, in his 1958--1959 seminar on transformation groups \cite{borel}.

The goal of these lectures is to give a quick introduction to equivariant cohomology in the context 
of algebraic geometry.  We will review the basic properties of $H_G^*X$ and give some examples of applications.

\bigskip

\noindent
{\it Ackowledgements.}  I learned much of what I know about equivariant cohomology from William Fulton, and the point of view presented here owes a debt to his lectures on the subject.  I am grateful to the organizers of IMPANGA for arranging the excellent conference in which these lectures took place.  These notes were assembled with the help of Piotr Achinger, and I thank him especially for assistance in typing, researching literature, and clarifying many points in the exposition.  I also thank the referee for valuable input and careful reading.

\subsection{The Borel construction}

Let $\Gg$ be a complex linear algebraic group, and let $X$
be a complex algebraic variety with a left $\Gg$-action.  
The construction Borel introduced in \cite{borel} goes as follows.  Find a contractible space $\EE\Gg$ with a free (right) $\Gg$-action.  (Such spaces exist, and are universal in an appropriate homotopy category; we will see concrete examples soon.)  Now form the quotient space
\[
  \EE \Gg \times^\Gg X := \EE \Gg \times X / (e\cdot g, x)\sim (e, g\cdot x).
\]

\begin{definition}\label{def:eq-coh}
The \define{equivariant cohomology} of $X$ (with respect to $G$) is the (singular) cohomology of $\E G \times^G X$:
\[
  H_G^*X := H^*(\E G\times^G X).
\]
(We always use singular cohomology with $\Z$ coefficients in these notes.)
\end{definition}

The idea behind the definition is that when $X$ is a \emph{free} $G$-space, we should have $H_G^*X = H^*(G\backslash X)$.  To get to this situation, we replace $X$ with a free $G$-space of the same homotopy type.  From a modern point of view, this is essentially the same as taking the cohomology of the quotient stack $[G\backslash X]$ (see, e.g., \cite{behrend}).

General facts about principal bundles ensure that this definition is independent of the choice of space $\E G$; for instance, $\BB \Gg = \EE \Gg/\Gg$ is unique up to homotopy.

When $X$ is a point, $\EE \Gg \times^\Gg \{\pt\} = \B \Gg$ usually has nontrivial topology, so $H_\Gg^*(\pt) \neq \Z$!  This is a key feature of equivariant cohomology, and $H_\Gg^*(\pt)=H^*\BB\Gg$ may be interpreted as the ring of characteristic classes for principal $\Gg$-bundles.  Other functorial properties are similar to ordinary cohomology, though.  (In fact, $H_\Gg^*(-)$ is a generalized cohomology theory, on a category of reasonable topological spaces with a left $\Gg$-action.)

\begin{example}
Let $\Gg=\C^*$.  The space $\EE\Gg = \CC^\infty\setminus\{0\}$ is contractible, and $\Gg$ acts freely, so $\BB G = \EE\Gg/\Gg = \PP^\infty$.  We see that
\[
  H_{\CC^*}^*(\pt) = H^*\PP^\infty \isom \ZZ[t],
\]
where $t = c_1(\Oo_{\PP^\infty}(-1))$ is the first Chern class of the tautological bundle.
\end{example}





\subsection{Approximation spaces}

The spaces $\EE \Gg$ and $\BB \Gg$ are typically infinite-dimensional, so they are not algebraic varieties.  (This may partly account for the significant lag before equivariant techniques were picked up by algebraic geometers.)  However, there are finite-dimensional, nonsingular algebraic varieties $\EE_m \to \BB_m = \EE_m/\Gg$ which serve as ``approximations'' to $\EE \Gg \to \BB \Gg$.  This works because of the following lemma:

\begin{lemma}\label{l:approx}
Suppose $\EE_m$ is any (connected) space with a free right $G$-action, and $H^i\EE_m = 0$ for $0<i<k(m)$ (for some integer $k(m)$).  Then for any $X$, there are natural isomorphisms
\[
  H^i(\EE_m \times^\Gg X) \isom H^i(\EE \Gg \times^\Gg X) =: H_\Gg^iX,
\]
for $i<k(m)$.
\end{lemma}

\begin{example}
For $G=\CC^*$, take $\EE_m = \CC^m\setminus\{0\}$, so $\BB_m = \PP^{m-1}$.  Since $\EE_m$ is homotopy-equivalent to the sphere $S^{2m-1}$, it satisfies the conditions of the lemma, with $k(m)=2m-1$ in the above lemma.  Note that $k(m)\to\infty$ as $m\to\infty$, so any given computation in $H_{\Gg}^*X$ can be done in $H^*(\EE_m\times^\Gg X)$, for $m\gg 0$.

We have $\BB_m = \PP^{m-1}$, so $\EE_m\to\BB_m$ is an intuitive choice for approximating $\EE\Gg\to \BB\Gg$.
\end{example}

\begin{example}\label{ex:Tn}
Similarly, for a torus $T\isom(\CC^*)^n$, take
\[
 \EE_m = (\CC^m \setminus \{0\})^{\times n}  \to (\PP^{m-1})^{\times n} = \BB_m.
\]
We see that $H_T^*(\pt) = H_T^*((\PP^\infty)^n) = \ZZ[t_1,\ldots,t_n]$, with $t_i = c_1(\Oo_i(-1))$.  (Here $\O_i(-1))$ is the pullback of $\O(-1)$ by projection on the $i$th factor.)
\end{example}

The above example is part of a general fact: For linear algebraic groups $G$ and $H$, one can always take $\EE(G\times H) = \EE G \times \EE H$.  Indeed, $G\times H$ acts freely (on the right) on the contractible space $\EE G \times \EE H$.

\begin{example}
Consider $\Gg = GL_n$, and let $\EE_m := M^\circ_{m\times n}$ be the set of full
rank $m\times n$ matrices, for $m>n$.  This variety is $k(m)$-connected, for 
$k(m) = 2(m-n)$.  Indeed, $M^\circ_{m\times n}$ is the complement of a closed 
algebraic set of codimension $(m-1)(n-1)$ in $M_{m\times n}$, and it is a general fact 
that $\pi_i(\CC^n \setminus Z)=0$ for $0<i\leq 2d-2$ if $Z$ is a Zariski closed subset 
of codimension $d$.  See \cite[\S A.4]{eq}.  It follows that the maps
\[
  \EE_m \to \BB_m = Gr(n, \CC^m)
\]
approximate $\EE\Gg\to\BB\Gg$, in the sense of Lemma~\ref{l:approx}.  We have
\[
  H^*_\Gg(\pt) = \ZZ[e_1, \ldots, e_n],
\]
where $e_i = c_i(S)$ is the $i$th Chern class of the tautological bundle over the Grassmannian.
\end{example}

Since any linear algebraic group $G$ embeds in some $GL_n$, the above example 
gives a construction of approximations $\EE_m$ that works for arbitrary $G$.

\begin{example}\label{ex:B}
The \emph{partial flag manifold} $\Fl(1,2,\ldots,n;\CC^m)$ parametrizes nested chains of subspaces $E_1 \subset E_2 \subset \cdots \subset E_n \subset \CC^m$, where $\dim E_i=i$.  There is also an infinite version, topologized as the limit taking $m\to\infty$.

If $B\subset GL_n$ is the subgroup of upper-triangular matrices, we have
\begin{diagram}
  \EE_m & = & M^\circ_{m\times n} \\
   \dTo &   &  \dTo \\
  \BB_m & = & \Fl(1,2,\ldots,n;\CC^m),
\end{diagram}
so $\BB B$ is the partial (infinite) flag manifold $\Fl(1,2,\ldots,n;\CC^\infty)$.
\end{example}

\begin{remark}
The idea of approximating the infinite-dimensional spaces $\EE\Gg$ and $\BB\Gg$ by finite-dimensional ones can be found in the origins of equivariant cohomology \cite[Remark XII.3.7]{borel}.  More recently, approximations have been used by Totaro and Edidin-Graham to define equivariant Chow groups in algebraic geometry.
\end{remark}

\subsection{Functorial properties}

Equivariant cohomology is functorial for \emph{equivariant maps}: given a homomorphism $G\xrightarrow{\phi} G'$ and a map $X \xrightarrow{f} X'$ such that $f(g\cdot x) = \phi(g)\cdot f(x)$, we get a pullback map
\[
  f^*\colon H_{G'}^*X' \to H_G^*X,
\]
by constructing a natural map $\E \times^G X \to \E' \times^{G'} X'$.

There are also \emph{equivariant Chern classes} and \emph{equivariant fundamental classes}:

\begin{itemize}
\item If $E \to X$ is an equivariant vector bundle, there are induced vector bundles $\EE_m \times^G E \to \EE_m \times^G X$.  Set
\[
  c^G_i(E) = c_i(\EE_m \times^G X) \in H_G^{2i}X = H^{2i}(\EE_m \times^G X),
\]
for $m\gg 0$.

\item When $X$ is a nonsingular variety, so is $\EE_m \times^G X$.  If $V \subseteq X$ is a $G$-invariant subvariety of codimension $d$, then $\EE_m \times^G V \subseteq \EE_m \times^G X$ has codimension $d$.  We define
\[
  [V]^G = [\EE_m \times^G V] \in H_G^{2d}X = H^{2d}(\EE_m \times^G X),
\]
again for $m\gg 0$.  (Any subvariety of a nonsingular variety defines a class, using e.g. Borel-Moore homology.)
\end{itemize}

In fact, the Chern classes could be defined directly as $c^G_i(E) = c_i(\E G\times^G E)$, but for $[V]^G$ one needs the approximation spaces.  (Of course, one has to check that these definitions are compatible for different $\E_m$'s.)

In the special case $X=\pt$, an equivariant vector bundle $E$ is the same as a representation of $\Gg$.  Associated to any representation, then, we have characteristic classes $c_i^\Gg(E) \in H^{2i}_\Gg (\pt)$.

\begin{example}\label{ex:lb-char}
Let $L_a=\CC$ be the representation of $\CC^*$ with the action
\[
  z\cdot v = z^a v
\]
where $a$ is a fixed integer.  Then 
\[
  \EE_m\times^{\CC^*} L_a \isom \Oo_{\PP^{m-1}}(-a)
\]
as line bundles on $\BB_m=\PP^{m-1}$, so $c_1^{\CC^*}(L_a) = at \in \ZZ[t]$.  (This also explains our choice of generator for $H_{\CC^*}^*(\pt)=\ZZ[t]$: we want $t$ to correspond to the standard representation, $L_1$.)
\end{example}

\begin{example}\label{ex:Cn}
Let $T = (\CC^*)^n$ act on $E=\CC^n$ by the standard action.  Then
\[
  c_i^T(E) = e_i(t_1, \ldots, t_n) \in H_T^*(\pt) = \ZZ[t_1, \ldots, t_n], \]
where $e_i$ is the $i$-th elementary symmetric function.  To see this, note that
\[
  \EE_m\times^T E \isom \Oo_1(-1)\oplus\ldots\oplus\Oo_n(-1)
\]
as vector bundles on $\BB_m=(\PP^{m-1})^n$.
\end{example}

\begin{problem}
Let $T$ be the maximal torus in $GL_n$, and let $E$ be an irreducible polynomial
$GL_n$-module.  The above construction assigns to $E$ its equivariant
Chern classes $c_i^T(E)$, which are \emph{symmetric} polynomials in variables
$t_1, \ldots, t_n$.  What are these polynomials?

Since they are symmetric polynomials, we can write
\[
 c_i^T(E) = \sum_\lambda a_\lambda s_\lambda(t),
\]
where $s_\lambda$ are the {\em Schur polynomials} (which are defined in \S\ref{s:lec3} below).  A theorem of Fulton and Lazarsfeld implies that the integers $a_\lambda$ are in fact {\em nonnegative}, as was observed by Pragacz in \cite[Corollary 7.2]{pragacz}.  This provides further motivation for the problem: we are asking for a combinatorial interpretation of the coefficients $a_\lambda$.
\end{problem}

\subsection{Fiber bundles} 

The formation of $\E\times^G X$ is really an operation of forming a fiber bundle with fiber $X$.  The map $X\xrightarrow{\rho}\pt$ becomes $\E\times^G X \xrightarrow{\rho} \B$ (via the first projection); ordinary cohomology $H^*X$ is an algebra over $H^*(\pt)=\Z$ (i.e., a ring), while $H_G^*X$ is an algebra over$H_G^*(\pt)$; and one can generally think of equivariant geometry as the geometry of bundles.  From this point of view, many statements about equivariant cohomology are essentially equivalent to things that have been known to algebraic geometers for some time---for instance, the Kempf-Laksov formula for degeneracy loci is the same as a ``Giambelli'' formula in $H_T^*\Gr(k,n)$.

\begin{example}[Equivariant cohomology of projective space]\label{ex:Pn}
Let $T = (\CC^*)^n$ act on $\CC^n$ in the usual way, defining an action on $\PP^{n-1}=\PP(\CC^n)$. 
This makes $\Oo_{\PP(\CC^n)}(1)$ a $T$-equivariant line bundle.  Write $\zeta = c_1^T(\Oo_{\PP^{n-1}}(1))$.

\begin{claim}
We have
\begin{align*}
 H_T^*\PP^{n-1} &\isom \ZZ[t_1,\ldots,t_n][\zeta]/(\zeta^n + e_1(t)\zeta^{n-1} + \cdots + e_n(t)) \\
 &= \ZZ[t_1,\ldots,t_n][\zeta]/( \textstyle{\prod_{i=1}^n(\zeta+t_i)} ).
\end{align*}
\end{claim}

\begin{proof}
Pass from the vector space $\CC^n$ to the vector bundle $E = \EE_m\times^T \CC^n$ on $\BB_m$.  We have
\[
  \EE_m\times^T \PP^{n-1} \isom \PP(E) \qquad \text{and} \qquad \EE_m\times^T\Oo_{\PP^{n-1}}(1) \isom \Oo_{\PP(E)}(1),
\]
all over $\BB_m$.  
The claim follows from the well-known presentation of $H^*\PP(E)$ over $H^*\BB$, since
\[
  e_i(t) = c^T_i(\CC^n) = c_i(E)  \qquad \text{and} \qquad \zeta = c^T_1(\Oo_{\PP^{n-1}}(1)) = c_1(\Oo_{\PP(E)}(1)),
\]
as in Example~\ref{ex:Cn}.
\end{proof} 
\end{example}

\subsection{Two notions} \label{ss:notions}

There are two general notions about equivariant cohomology to have in mind, especially when $G$ is a torus.


The first notion is that \emph{equivariant cohomology determines ordinary cohomology}.  From the fiber bundle picture, with the commutative diagram
\begin{diagram}
 X & \rInto & \EE\times^\Gg X \\
 \dTo &    &  \dTo  \\
\pt & \rInto & \BB,
\end{diagram}
$H_G^*X$ is an algebra over $H_G^*(\pt)$, and restricting to a fiber gives a canonical map $H_G^*X \to H^*X$, compatible with $H_G^*(\pt) \to H^*(\pt) = \Z$.  In nice situations, we will have:
\begin{quote}
 $H_G^*X \to H^*X$ is {\em surjective}, with kernel generated by the kernel of $H_G^*(\pt) \to \Z$.
\end{quote} 


The second notion is that \emph{equivariant cohomology is determined by information at the fixed locus}. 
By functoriality, the inclusion of the fixed locus $\iota\colon X^G \hookrightarrow X$ gives a restriction (or ``localization'') map $\iota^*\colon H_G^*X \to H_G^*X^G$.  In nice situations, we have:
\begin{quote}
$\iota^*\colon H_G^*X \to H_G^*X^G$ is {\em injective}.
\end{quote} 


\begin{example}
Both of these notions can fail, even when $G$ is a torus.  For example, take 
$\Gg = X = \CC^*$, where $\Gg$ acts on itself via left
multiplication.  One the one hand,
\[
  H^1_{\CC^*} (\CC^*) = H^1((\CC^\infty\setminus\{0\})\times^{\CC^*} \CC^*) = H^1(\CC^\infty\setminus\{0\}) = 0,
\]
but on the other hand, $H^1(\CC^*) =H^1(S^1) = \ZZ$, so the first notion fails.  
Since the action has no fixed points, the second cannot hold, either.  However, we will see that the ``nice'' situations, where both notions do hold, include many interesting cases.
\end{example}

When the second notion holds, it provides one of the most powerful techniques in equivariant theory.  To get an idea of this, suppose $X$ has \emph{finitely} many fixed points; one would never expect an injective restriction map $\iota^*$ in ordinary cohomology, by degree reasons!  Yet in many situations, all information about $H_T^*X$ is contained in the fixed locus.  This will be the topic of the next lecture.

\section{Lecture Two: Localization}\label{s:lec2}



From now on, we will consider only tori: $G=T \isom (\C^*)^n$.  Since it comes up often, it is convenient introduce notation for the equivariant cohomology of a point:
\[
  \Lambda = \Lambda_T = H_T^*(\pt) \isom \Z[t_1,\ldots,t_n].
\]

\subsection{Restriction maps}

If $X$ is a $T$-space and $p\in X$ is a fixed point, then the inclusion
$\iota_p \colon \{p\}\to X$ is equivariant, so it induces a map on equivariant 
cohomology
\[
  \iota_p^\ast \colon H^*_T X \to X^*_T(\{p\}) = \Lambda.
\]


\begin{example}
Let $E$ be an equivariant vector bundle of rank $r$ on $X$, with $E_p$ the fiber at $p$.  Then $\iota_p^*(c^T_i(E)) = c^T_i(E_p)$, as usual.  Now $E_p$ is just a representation of $T$, say with weights (characters) $\chi_1,\ldots,\chi_r$.  That is, $E_p \isom \CC^r$, and $t\cdot(v_1,\ldots,v_r) = (\chi_1(t)v_1,\ldots,\chi_r(t)v_r)$ for homomorphisms $\chi_i\colon T\to\C^*$.  So $c^T_i(E_p) = e_i(\chi_1,\ldots,\chi_r)$ is the $i$th elementary symmetric polynomial in the $\chi$'s.

In particular, the top Chern class goes to the product of the weights of $E_p$:
\[
  \iota_p^*(c_r(E)) = \chi_1\cdots\chi_r.
\]
\end{example}



\begin{example}\label{ex:Pn-res}
Consider $X = \PP^{n-1}$ with the standard action of $T=(\CC^*)^n$, so
\[
  (t_1, \ldots, t_n)\cdot [x_1,\ldots,x_n] = [t_1 x_1 , \ldots , t_n x_n].
\]
The fixed points of this action are the points
\[
  p_i = [\underbrace{0,\ldots,0,1}_i ,0,\ldots,0] \quad \text{for} \quad i = 1, \ldots, n.
\]
The fiber of the tautological line bundle $\Oo(-1)$ at $p_i$ is the coordinate line $\CC\cdot \epsilon_i$, so $T$ acts on $\Oo(-1)_{p_i}$ by the character $t_i$, and on $\Oo(1)_{p_i}$ by $-t_i$.  In the notation of Example~\ref{ex:lb-char}, $\Oo(-1)_{p_i}\isom L_{t_i}$ and $\Oo(1)_{p_i}\isom L_{-t_i}$.  Setting $\zeta = c^T_1(\Oo(1))$, we see that 
\[
  \iota_{p_i}^*\zeta = -t_i.
\]
\end{example}


\begin{exercise}
Show that the map of $\Lambda$-algebras
\begin{diagram}
  \Lambda[\zeta]/\prod_{i=1}^n(\zeta+t_i) & \rTo & \Lambda^{\oplus n} \\
    \zeta       & \rMapsto     & (-t_1,\ldots,-t_n)
\end{diagram}
is injective.  (By Examples~\ref{ex:Pn} and \ref{ex:Pn-res}, this is the restriction map $H_T^*\P^{n-1} \xrightarrow{\iota^*} H_T^*(\P^{n-1})^T = \bigoplus H_T^*(p_i)$.)
\end{exercise}

\subsection{Gysin maps}


For certain kinds of (proper) maps $f\colon Y \to X$, there are Gysin pushforwards
\[
  f_*\colon H_T^*Y \to H_T^{*+2d}X,
\]
as in ordinary cohomology.  Here $d=\dim X - \dim Y$.  We'll use two cases, always assuming $X$ and $Y$ are nonsingular varieties.

\begin{enumerate}
 \item {\it Closed embeddings.}  If $\iota\colon Y\hookrightarrow X$ is a $T$-invariant closed embedding of codimension $d$, we have $\iota_*\colon H_T^*Y \to H_T^{*+2d}X$.  This satisfies:
  \begin{enumerate}
   \item $\iota_*(1) = \iota_*[Y]^T = [Y]^T$ is the fundamental class of $Y$ in $H_T^{2d}X$.
   \item (self-intersection) $\iota^*\iota_*(\alpha) = c^T_d(\Nn_{Y/X})\cdot \alpha$, where $\Nn_{Y/X}$ is the normal bundle.
  \end{enumerate}
 \item {\it Integral.}  For a complete (compact) nonsingular variety $X$ of dimension $n$, the map $\rho\colon X\to\pt$ gives $\rho_*\colon H_T^*X \to H_T^{*-2n}(\pt)$.
\end{enumerate}

\begin{example}\label{ex:P1}
Let $T$ act on $\PP^1$ with weights $\chi_1$ and $\chi_2$, so $t\cdot [x_1,\,x_2] = [\chi_1(t) x_1
,\, \chi_2(t) x_2]$). Let $p_1=[1,\,0]$, $p_2=[0,\,1]$ as before.  Setting $\chi = \chi_2 - \chi_1$, 
the tangent space $T_{p_1} \PP^1$ has weight $\chi$:
\[
  t\cdot [1,\, a] = [\chi_1(t),\, \chi_2(t)a] = [1,\, \chi(t)a].
\]
Similarly, $T_{p_2} \PP^1$ has weight $-\chi$.  
So $\iota_{p_1}^*[p_1]^T = c_1^T(T_{p_1}\P^1) = \chi$, and $\iota_{p_2}[p_2]^T = -\chi$.  (And the other restrictions are zero, of course.)  From Example~\ref{ex:Pn-res}, we know $\iota_{p_1}^*\zeta = -\chi_1 = \chi-\chi_2$ and $\iota_{p_1}^*\zeta = -\chi_2 = -\chi-\chi_1$, so
\[
  [p_1]^T = \zeta+\chi_2 \qquad \text{ and } \qquad [p_2]^T = \zeta+\chi_1
\]
in $H_T^*\PP^1$.
\end{example}

\begin{exercise}
More generally, show that if $T$ acts on $\PP^{n-1}$ with weights $\chi_1, \ldots, \chi_n$, then
\[
  [p_i]^T = \prod_{j\neq i} (\zeta + \chi_j)
\]
in $H_T^*\PP^{n-1}$.
\end{exercise}

\begin{example}\label{ex:self-int}
If $p\in Y \subseteq X$, with $X$ nonsingular, and $p$ a nonsingular point on the (possibly singular) subvariety $Y$, then
\[
  \iota_p^*[Y]^T = c_d^T(N_p) = \prod_{i=1}^d \chi_i
\]
in $\Lambda=H_T^*(p)$, where the $\chi_i$ are the weights on the normal space $N_p$ to $Y$ at $p$.
\end{example}

\subsection{First localization theorem}


Assume that $X$ is a nonsingular variety, with finitely many fixed points.  Consider the sequence of maps
\[
  \bigoplus_{p\in X^T} \Lambda = H_T^*X^T \xrightarrow{\iota_*} H_T^*X \xrightarrow{\iota^*} H_T^*X^T = \bigoplus_{p\in X^T} \Lambda.
\]
The composite map $\iota^*\iota_*\colon\bigoplus\Lambda \to \bigoplus\Lambda$ is diagonal, and is multiplication by $c^T_n(T_pX)$ on the summand corresponding to $p$.

\begin{theorem}\label{t:one}
Let $S\subseteq \Lambda$ be a multiplicative set containing the element
\[
  c := \prod_{p\in X^T} c_n^T(T_p X).
\]
\begin{enumerate}[(a)]
\item \label{ta} The map 
\[
  S^{-1}\iota^* \colon S^{-1} H^\ast_T X \to S^{-1} H_T^\ast X^T \eqno (*)
\]
is surjective, and the cokernel of $\iota_*$ is annihilated by $c$.

\item \label{tb} Assume in addition that $H_T^\ast X$ is a free $\Lambda$-module of rank at most $\#X^T$.
Then the rank is equal to $\#X^T$, and the above map (*) is an isomorphism.
\end{enumerate}
\end{theorem}

\begin{proof}For \eqref{ta}, it suffices to show that the composite map $S^{-1}(\iota^* \circ \iota_*) = S^{-1}\iota^*\circ S^{-1}\iota_*$ is surjective.  This in turn follows from the fact that the determinant
\[
  \det (\iota^*\circ \iota_*) = \prod_p c_n^T(T_p X) = c
\]
becomes invertible after localization.

For \eqref{tb}, surjectivity of $S^{-1}\iota^*$ implies $\rank H_T^*X \geq \#X^T$, and hence equality.  Finally, since $S^{-1}\Lambda$ is noetherian, a surjective map of finite free modules of the same rank is an isomorphism.
\end{proof}

\subsection{Equivariant formality}

The question arises of how to verify the hypotheses of Theorem~\ref{t:one}.  To this end, we consider the following condition on a $T$-variety $X$:
\begin{itemize}
\item[(EF)] $H^*_TX$ is a free $\Lambda$-module, and has a $\Lambda$-basis that restricts to a $\Z$-basis for $H^*X$.
\end{itemize}
Using the Leray-Hirsch theorem, this amounts to degeneration of the Leray-Serre spectral sequence of the fibration $\EE\times^T X \to \BB$.  A space satisfying the condition (EF) is often called \define{equivariantly formal}, a usage introduced in \cite{gkm}.

One common situation in which this condition holds is when $X$ is a nonsingular projective variety, with $X^T$ finite.  In this case, the Bia{\l}ynicki-Birula decomposition yields a collection of $T$-invariant subvarieties, one for each fixed point, whose classes form a $\Z$-basis for $H^*X$.  The corresponding equivariant classes form a $\Lambda$-basis for $H_T^*X$ restricting to the one for $H^*X$, so (EF) holds.  Moreover, since the basis is indexed by fixed points, $H_T^*X$ is a free $\Lambda$-module of the correct rank, and assertion \eqref{tb} of Theorem~\ref{t:one} also holds.

\begin{corollary}
The ``two notions'' from \S\ref{ss:notions} hold for a nonsingular projective $T$-variety with finitely many fixed points:
\[
  H_T^*X \twoheadrightarrow H^*X \qquad \text{ and } \qquad H_T^*X \hookrightarrow H_T^*X^T.
\]
\end{corollary}

\begin{remark}
Condition (EF) is not strictly necessary to have an injective localization map $S^{-1}H_T^*X \to S^{-1}H_T^*X^T$.  In fact, if one takes $S=\Lambda\setminus\{0\}$, no hypotheses at all are needed on $X$: this map is always an isomorphism (though the rings may become zero); see \cite[\S IV.1]{hsiang}.

On the other hand, this phenomenon is peculiar to torus actions.  For example, the group $B$ of upper-triangular matrices acts on $\PP^{n-1}$ with only one fixed point.  However, since $B$ admits a deformation retract onto the diagonal torus $T$, the ring $H_B^*\PP^{n-1}=H_T^*\PP^{n-1}$ is a free module over $\Lambda_B=\Lambda_T$.  There can be no injective map to $H_B^*(\PP^{n-1})^B=H_B^*(\pt)=\Lambda$, even after localizing.  (The difference is that $H_B^*(X\setminus X^B)$ is not necessarily a torsion $\Lambda_B$-module when $B$ is not a torus.)
\end{remark}

\subsection{Integration formula (Atiyah-Bott-Berline-Vergne)}

\begin{theorem}\label{t:two}
Let $X$ be a compact nonsingular variety of dimension $n$, with finitely many $T$-fixed points.  Then
\[
  \rho_*\alpha = \sum_{p\in X^T} \frac{\iota_p^*\alpha}{c^T_n T_pX}
\]
for all $\alpha \in H_T^*X$.  
\end{theorem}

\begin{proof}
Since $\iota_*\colon S^{-1}H_T^*X^T \to H_T^*X$ is surjective, it is enough to assume $\alpha = (\iota_p)_*\beta$, for some $\beta \in H_T^*(p) = \Lambda$.  Then the LHS of the displayed equation is $\rho_*\alpha = \rho_*(\iota_p)_*\beta = \beta$.  (The composition $H_T^*(p) \xrightarrow{(\iota_p)_*} H_T^*X \xrightarrow{\rho_*} \Lambda$ is an isomorphism.)  The RHS is
\[
  \sum_{q \in X^T} \frac{\iota_q^*(\iota_p)_*\beta}{c^T_n T_qX} = \frac{\iota_p^*(\iota_p)_*\beta}{c^T_n T_pX} = \beta,
\]
using the self-intersection formula for the last equality.
\end{proof}

\begin{example}
Take $X=\P^{n-1}$, with the standard action of $T$ via character $t_1,\ldots, t_n$, and let $\zeta = c^T_1(\O(1))$.  Then one computes
\begin{align*}
  \rho_*(\zeta^k) &= \left\{\begin{array}{cl} 
  0 & \text{ if } k<n-1, \text{ by degree: } H_T^{2k-2(n-1)}(\pt) = 0; \\
  1 & \text{ if } k=n-1, \text{ by ordinary cohomology}. \end{array}\right. \\
\intertext{
On the other hand, using the localization formula, we obtain
}
  \rho_*(\zeta^k) &= \sum_{i=1}^n \frac{(-t_i)^k}{\prod_{j\neq i} (t_j-t_i)},
\end{align*}
yielding a nontrivial algebraic identity!
\end{example}

\begin{remark}
Ellingsrud and Str{\o}mme \cite{ell-str} used this technique, with the aid of computers, to find the answers to many difficult enumerative problems, e.g., the number of twisted cubics on a Calabi-Yau three-fold.

As an illustrative exercise, one could compute the number of lines passing through to four given lines in $\P^3$.  (Use localization for the action of $T\isom(\C^*)^4$ on the Grassmannian $\Gr(2,4)$.)  A more challenging problem is to compute number of conics tangent to five given conics in $\P^2$, via localization for the action of $T\isom(\C^*)^2$ on the space of complete conics; see, e.g., \cite[p.~15]{brion-eit}.
\end{remark}

\subsection{Second localization theorem}\label{ss:gkm}

A remarkable feature of equivariant cohomology is that one can often characterize the image of the restriction map $\iota^*\colon H_T^*X \to H_T^*X^T$, realizing $H_T^*X$ as a subring of a direct sum of polynomial rings.  To state a basic version of the theorem, we use the following hypothesis.  For characters $\chi$ and $\chi'$ appearing as weights of $T_p X$, for $p\in X^T$, assume:
\begin{enumerate}
\renewcommand{\theenumi}{(*)}
\item  if $\chi$ and $\chi'$ occur in the \emph{same} $T_pX$ they are relatively prime in $\Lambda$. \label{gkm:hypo}
\end{enumerate}
\renewcommand{\theenumi}{\arabic{enumi}}
Condition \ref{gkm:hypo} implies that for each $p\in X^T$ and each $\chi$ occurring as a weight for $T_pX$, there exists a unique $T$-invariant curve $E=E_{\chi,p}\isom\P^1$ through $p$, with $T_pE \isom L_\chi$.  By Example~\ref{ex:P1}, $E^T = \{p,q\}$, and $T_qE$ has character $-\chi$.

\begin{theorem}\label{t:gkm}
Let $X$ be a nonsingular variety with $X^T$ finite, and assume \ref{gkm:hypo} holds.  Then an element
\[
  (u_p) \in \bigoplus_{p\in X^T} \Lambda = H_T^*X^T
\]
lies in $\iota^*H_T^*X$ if and only if, for all $E = E_{\chi,p} = E_{-\chi,q}$, the difference $u_p-u_q$ is divisible by $\chi$.
\end{theorem}

As with the other localization theorems, the idea of the proof is to use the Gysin map and self-intersection formula, this time applied to the compositions
\[
 H_T^*X^T \to H_T^*X^\chi \to H_T^*X \to H_TX^\chi \to H_T^*X^T,
\]
where $X^\chi$ is a union of invariant curves $E_\chi$ for a fixed character $\chi$.  For a detailed proof, see \cite[\S5]{eq}.

The theorem is originally due to Chang and Skjelbred \cite{chang-skjelbred}.  It was more recently popularized in an algebraic geometry context by Goresky-Kottwitz-MacPherson \cite{gkm}.  The utility of this characterization is that it makes $H_T^*X$ a combinatorial object: the ring can be computed from the data of a graph whose vertices are the fixed points $X^T$ and whose edges are the invariant curves $E=E_{\chi,p}=E_{-\chi,q}$.

\section{Lecture Three: Grassmannians and Schubert calculus}\label{s:lec3}

\subsection{Pre-history: Degeneracy loci}

Let $X$ be a Cohen-Macaulay variety, and let $E$ be a rank $n$ vector bundle on $X$,
admitting a filtration 
\[ 0=E_0\subset E_1\subset\ldots\subset E_n = E, \quad \text{where} \quad \rank E_i = i. \]
Let $F$ be a rank $r$ vector bundle, and let $\phi\colon E\to F$ be a surjective morphism.  Given
a partition $\lambda = (r\geq \lambda_1 \geq \lambda_2\geq\ldots\geq \lambda_{n-r} \geq 0)$, the associated \emph{degeneracy locus} is defined as
\[
  D_\lambda(\phi) = \{ x\in X \,|\, \rank( E_{r-\lambda_i + i}(x)\xrightarrow{\phi(x)} F(x) ) \leq r-\lambda_i \text{ for } 1\leq i\leq n-r\} \subseteq X.
\]
Since these schemes appear frequently in algebraic geometry, it is very useful to have a formula for their cohomology classes.  Such a formula was given by Kempf and Laksov \cite{KempfLaksov1974}, and independently by Lascoux \cite{lascoux}:

\begin{theorem}[Kempf-Laksov]
Set $k=n-r$.  When $\codim D_\lambda = |\lambda| := \sum \lambda_i$, we have
\begin{align*}
  [D_\lambda] = \Delta_\lambda(c(F-E)) &:= \det( c_{\lambda_i + j-i}(i) ) \\ &\;= 
\left|
\begin{array}{cccc}
{c_{\lambda_1}(1) } & c_{\lambda_1+1}(1) & \cdots & c_{\lambda_1 + k-1}(1)\\
c_{\lambda_2 - 1}(2) & c_{\lambda_2}(2)   &   & \vdots \\
\vdots &   & \ddots & \vdots \\
c_{\lambda_{k} - k+1}(k) & \cdots & \cdots & c_{\lambda_{k}}(k)\\
\end{array} \right|
\end{align*}
in $H^*X$, where $c_p(i) = c_p(F - E_{r-\lambda_i + i})$.
\end{theorem}

\noindent
Here the notation $c(A-B)=c(A)/c(B)$ means the formal series expansion of $(1+c_1(A)+c_2(A)+\cdots)/(1+c_1(B)+c_2(B)+\cdots)$, and $c_p$ is the degree $p$ term.

The proof of this theorem starts with a reduction to the Grassmannian bundle $\pi\colon\GGr(k, E)\to X$.  Since $\phi$ is surjective, the subbundle $K=\ker(\phi)\subseteq E$ has rank $k$, and it defines a section $\sigma\colon X\to \GGr(k,E)$ such that $\sigma^*S=K$.  (Here $S\subseteq\pi^*E$ is the tautological rank $k$ bundle on $\GGr(k,E)$.)

In fact, the theorem is equivalent to a formula in equivariant cohomology.  The universal base for rank $n$ bundles with filtration is $\BB{B}$; see Example~\ref{ex:B}.  Consider the following diagram:
\begin{diagram}
  \GGr(k,E)   & \rTo   &  \GGr(k,\Ee) \\
   \dTo^\pi   & \ruDashto^g  &    \dTo      \\
    X         & \rTo^f &   \BB B.
\end{diagram}
Writing $\Ee_1 \subset \Ee_2 \subset \cdots \subset \Ee_n=\Ee$ for the tautological sequence on $\BB B=\Fl(1,2,\ldots,n;\CC^\infty)$, the map $f$ is defined by $E_i=f^*\Ee_i$ for $1\leq i\leq n$.  The map $g$ is defined by $F=g^*\Ff$, where $\Ee \to \Ff$ is the universal quotient on $\GGr(k,\Ee)$.  Now a formula for $D_\lambda$ in $H^*X$ can be pulled back from a universal formula for a corresponding locus $\OOmega_\lambda$ in $H^*\GGr(k,\Ee)=H_B^*\Gr(k,\CC^n)=H_T^*\Gr(k,\CC^n)$.  We will see how such a formula can be deduced combinatorially, using equivariant localization.

\begin{remark}
Some extra care must be taken to ensure that the bundles $E_i$ and $F$ are pulled back from the algebraic approximations $\BB_m$; see, e.g., \cite[p.~486]{gr-diag}.
\end{remark}

\subsection{The basic structure of Grassmannians}

The Grassmannian $X=\Gr(k, \CC^n)$ is the space of $k$-dimensional linear subspaces in $\CC^n$
and can be identified with the quotient
\[
  M_{n\times k}^\circ / GL_k
\]
of the set of full-rank $n$ by $k$ matrices by the action of $GL_k$ by right multiplication.  
The groups $T \isom (\CC^\ast)^n\subset B \subset GL_n$ act on $X$ by
left multiplication.

For any $k$-element subset $I\subset\{1, \ldots, n\}$, we denote by $U_I$ the set of all
$k$-dimensional linear subspaces of $\CC^n$ whose projection on the subspace spanned by
the vectors $\{e_i \,|\, i\in I\}$, where $\{e_1, \ldots, e_n\}$ is the standard basis for $\CC^n$.  
It follows that $U_I \isom \CC^{k(n-k)}$ and $U_I$ is open in $X$.  Indeed, $U_I$
can be identified with the set of $k\times (n-k)$ matrices $M$ whose square submatrix on rows $I$ is
equal to the $k\times k$ identity matrix.  For $I = \{2, 4, 5\}$, we have
\[ 
U_I = \left\{ \left[
\begin{array}{ccc}
* & * & * \\
1 & 0 & 0 \\
* & * & * \\
0 & 1 & 0 \\
0 & 0 & 1 \\
* & * & * \\
\end{array} \right] \right\}. 
\]
In particular, $\dim X = k(n-k)$.

The topology of $Gr(k, \CC^n)$ is easily studied by means of the decomposition into \emph{Schubert
cells}: 
Each point $p\in X$ has an ``echelon'' form, that is, it can be represented by a full rank $n\times k$ matrix such as
\[ 
\left[
\begin{array}{ccc}
* & * & * \\
1 & 0 & 0 \\
* & * & 0 \\
0 & 1 & 0 \\
0 & 0 & 1 \\
0 & 0 & 0 \\
\end{array} \right].
\]
We denote the set of rows with $1$'s by $I$ (so in the example above, $I = \{2, 4, 5\}$)
and call it the \emph{pivot} of the corresponding subspace.  For any $k$-element subset $I\subset \{1, \ldots, n\}$, the set $\Omega^\circ_I$ of points with pivot $I$ forms a cell, isomorphic to an affine space of dimension equal to the number of stars in the matrix.  These are called \emph{Schubert cells}, and they give a cell decomposition of $X$.

Summarizing, we have
\[
  \Omega_I^\circ\stackrel{\text{closed}}{\hookrightarrow} U_I \stackrel{\text{open}}{\hookrightarrow} X.
\]
Note that $\Omega^\circ_I$ and $U_I$ are $T$-stable.

\subsection{Fixed points and weights}

Let $p_I \in U_I$ be the origin, that is, the point corresponding
to the subspace spanned by $\{e_i \,|\, i\in I\}$.  Working with matrix representatives, the following basic facts are easy to prove:
sees that
\begin{itemize}
\item The $T$-fixed points in $X$ are precisely the points $p_I$.  In particular, $\#X^T = \binom{n}{k}$.
\item The weights of $T$ on $T_{p_I} X = T_{p_I} U_I \isom U_I$ are $\{t_j - t_i \,|\, i\in I, j\notin I\}$.
\item The weights of $T$ on $T_{p_I} \Omega_I^\circ $ are $\{t_j - t_i \,|\, i\in I, j\notin I, i > j\}$.
\item The weights of $T$ on $N_{\Omega_I/X, p_I}$ are $\{t_j - t_i \,|\, i\in I, j\notin I, i < j\}$.
\end{itemize}

\begin{example}
With $I=\{2,4,5\}$, one sees that $t_4-t_3$ is a weight on $T_{p_I}\Omega^\circ_I$:
\[
z\cdot 
\left[
\begin{array}{ccc}
0 & 0 & 0 \\
1 & 0 & 0 \\
0 & a & 0 \\
0 & 1 & 0 \\
0 & 0 & 1 \\
0 & 0 & 0 \\
\end{array} \right]
=
\left[
\begin{array}{ccc}
0 & 0 & 0 \\
z_2 & 0 & 0 \\
0 & z_3 a & 0 \\
0 & z_4 & 0 \\
0 & 0 & z_5 \\
0 & 0 & 0 \\
\end{array} \right] 
\equiv
\left[
\begin{array}{ccc}
0 & 0 & 0 \\
1 & 0 & 0 \\
0 & \frac{z_3}{z_4} a & 0 \\
0 & 1 & 0 \\
0 & 0 & 1 \\
0 & 0 & 0 \\
\end{array} \right] .
\]
The other weights can be determined in a similar manner.
\end{example}

\subsection{Schubert classes in $H_T^*X$}

The closure $\Omega_I:= \overline{\Omega^\circ_I}$ of a Schubert cell is called a \emph{Schubert variety}.
It is a disjoint
union of all Schubert cells $\Omega_J^\circ$ for $J\leq I$ with respect to the \emph{Bruhat order}:
\[
  J\leq I \text{ iff } j_1\leq i_1,\, j_2\leq i_2,\, \ldots,\, j_k \leq i_k .
\]

Since the Schubert cells have even (real) dimension, it follows that the classes of their closures 
form bases for $H^*X$ and $H^*_TX$:
\[
  H^*X = \bigoplus_I \ZZ\cdot [\Omega_I] \quad\text{ and }\quad H^*_T X  = \bigoplus_I \Lambda \cdot [\Omega_I]^T .
\]
In particular, $H_T^*X$ is free over $\Lambda$, of the correct rank, so the localization theorems apply.

Let us record two key properties of Schubert classes.  From the description of weights, and from Example~\ref{ex:self-int}, it follows that
\begin{equation}
  \iota^*_{p_I} [\Omega_I]^T = \prod_{i\in I \, j\notin I\, i < j} (t_j - t_i). \label{eq:one}
\end{equation}
Moreover, since $p_J\in \Omega_I$ if and only if $J\leq I$ we see
that
\begin{equation}
  \iota^*_{p_J} [\Omega_I]^T = 0 \quad\text{ if } J\not\leq I. \label{eq:two}
\end{equation}
It turns out that the Schubert classes are unique solutions to a corresponding interpolation problem:

\begin{proposition}[Knutson-Tao \cite{KnutsonTao2003}, Feh\'er-Rim\'anyi \cite{FeherRimanyi2003}]\label{p:kt}
Relations \eqref{eq:one} and \eqref{eq:two} determine $[\Omega_I]^T$ as a class in $H_T^* X$. 
\end{proposition}

\subsection{Double Schur functions}

Every $k$-element subset $I\subset\{1, \ldots, n\}$ can be represented in a form of
a \emph{Young diagram}: one first draws a rectangle with $k$ rows and $n-k$ columns,
then draws a path from the upper-right to the lower-left corner; at $i$-th step,
we go downwards if $i\in I$ and to the left otherwise.  Counting the number of boxes in each row, we obtain a partition $\lambda = \lambda(I) = (\lambda_1 \geq \lambda_2 \geq \cdots \geq \lambda_k\geq 0)$.

\begin{example}
With $n=7$ and $k=3$, the subset $I=\{2,4,5\}$ corresponds to the partition $\lambda= (3, 2, 2)$,
whose Young diagram is below: 

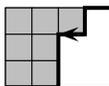
\begin{center}
\pspicture(-100,-5)(150,40)
\psline{-}(0,10)(0,40)(40,40)(40,10)(0,10)
\graybox(0,30)
\graybox(10,30)
\graybox(20,30)
\graybox(0,20)
\graybox(10,20)
\graybox(0,10)
\graybox(10,10)
\psline[linewidth=1.5pt]{-}(0,10)(20,10)(20,30)(30,30)(30,40)(40,40)
\psline[linewidth=1.5pt]{<-}(20,30)(30,30)
\endpspicture
\end{center}
\end{example}

A \emph{semistandard Young tableau} (SSYT for short) on a Young diagram $\lambda$ 
is a way of filling the boxes of $\lambda$ with numbers from $\{1, \ldots, k\}$
such that they are weakly increasing along rows (going to the right) and strictly increasing along 
columns (going downwards).  We write $SSYT(\lambda)$ for the set of all SSYT on the diagram $\lambda$.

\begin{definition}
The \define{double Schur function} associated to $\lambda$ is a polynomial in two sets of variables,
$x = (x_1, x_2, \ldots)$ and $u = (u_1, u_2, \ldots)$, defined by the formula
\[
  s_\lambda(x | u) = \sum_{S\in SSYT(\lambda)} \prod_{(i, j)\in\lambda} \left( x_{S(i,j)} - u_{S(i,j)+j-i}\right).
\]
Here $S(i,j)$ is the $i,j$-entry of the tableau $S$, using matrix coordinates on the diagram $\lambda$.
\end{definition}

\begin{example}
There are $8$ semistandard Young tableaux on the diagram $\lambda =$\blankbox(0,3)\blankbox(10,3)\blankbox(0,-7) \rput(25,0){,}
\[ 
\begin{Young}
1 & 1 \cr 2 \cr
\end{Young}\quad  
\begin{Young}
1 & 1 \cr 3 \cr
\end{Young}\quad  
\begin{Young}
1 & 2 \cr 2 \cr
\end{Young}\quad  
\begin{Young}
1 & 2 \cr 3 \cr
\end{Young}\quad  
\begin{Young}
1 & 3 \cr 2 \cr
\end{Young}\quad  
\begin{Young}
1 & 3 \cr 3 \cr
\end{Young}\quad
\begin{Young}
2 & 2 \cr 3 \cr
\end{Young}\quad  
\begin{Young}
2 & 3 \cr 3 \cr
\end{Young},
\]
so the double Schur function is
\begin{align*}
  s_\lambda(x|u) &= (x_1-u_1)(x_1-u_2)(x_2-u_1) + (x_1-u_1)(x_1-u_2)(x_3-u_2) \\
                 &\quad +(x_1-u_1)(x_2-u_3)(x_2-u_1) + (x_1-u_1)(x_2-u_3)(x_3-u_2) \\
                 &\quad +(x_1-u_1)(x_3-u_4)(x_2-u_1) + (x_1-u_1)(x_3-u_4)(x_3-u_2) \\
                 &\quad +(x_2-u_2)(x_2-u_3)(x_3-u_2) + (x_2-u_2)(x_3-u_4)(x_3-u_2) .
\end{align*}
\end{example}

It is not obvious from the definition, but in fact $s_\lambda(x|u)$ is \emph{symmetric} in the $x$ variables.  A nice discussion of some properties of these functions can be found in \cite[6th Variation]{mac}; see also \cite{cll} for a more recent study.  A crucial fact for us is that they solve the same interpolation problem as the Schubert classes:

\begin{proposition}[Molev-Sagan \cite{MolevSagan1999}, Okounkov-Olshanski \cite{OkounkovOlshanski1997}]\label{p:ms}
Set
\[
  u_i = - t_{n+1-i}, \quad t_j^\lambda = - t_{i_j},
\]
where $I = \{i_1 < \cdots < i_k\}$ is the subset corresponding to $\lambda$.  Then
\begin{align*}
s_\lambda(t^\lambda| u) &= \prod_{i\in I\, j\notin I\,i<j} (t_j - t_i)  \\
\intertext{and}
s_\lambda(t^\mu | u) &= 0 \quad \text{ if }\mu\not\supseteq \lambda.
\end{align*}
\end{proposition}

One can deduce a ``Giambelli'' formula for the Schubert classes.  When $\lambda$ is the partition corresponding to a subset $I$, let us write $\Omega_\lambda = \Omega_I$ for the corresponding Schubert variety.  Now observe that the stars in the echelon matrix form of $\Omega_\lambda^\circ$ naturally fit into the \emph{complement} of the diagram $\lambda$ inside the $k\times (n-k)$ rectangle.  Therefore the codimension of $\Omega_\lambda$ is equal to $|\lambda|=\sum_{i=1}^k \lambda_i$, the number of boxes in $\lambda$.  Moreover, $J\leq I$ with respect to the Bruhat order if and only if $\lambda \supseteq \mu$ as diagrams, where $\mu$ is the partition corresponding to $J$.

\begin{corollary}
Let $S$ be the tautological subbundle on $X=\Gr(k,\CC^n)$, let $x_1,\ldots,x_k$ be the equivariant Chern roots of the dual bundle $S^\vee$, and let $t_1,\ldots,t_n$ be the standard weights of $T=(\CC^*)^n$.  Using the substitution from Proposition~\ref{p:ms}, we have
\[
  [\Omega_\lambda]^T = s_\lambda(x|u)
\]
in $H_T^*X$.
\end{corollary}

\begin{proof}
Using Propostions~\ref{p:kt} and \ref{p:ms}, it suffices to observe that the weights of $S|_{p_J}$ are $t^\mu_1,\ldots,t^\mu_k$ if $\mu$ is the partition corresponding to $J$, so $s_\lambda(t^\mu|u)$ is equal to the restriction $\iota_J^*s_\lambda(x|u)$.
\end{proof}

Finally, to relate this formula to the Kempf-Laksov formula, consider the sequence of vector bundles on $X$
\[
  E_1 \hookrightarrow E_2 \hookrightarrow \cdots \hookrightarrow E_n = \CC^n \xrightarrow{\phi} Q,
\]
where the $E_i$ are trivial bundles, spanned by $\{e_1,\ldots,e_i\}$, and $Q=\CC^n/S$ is the universal quotient bundle.  We claim that
\[
  \Omega_\lambda = D_\lambda(\phi).
\]
This is a standard fact, and the main point is that $D_\lambda(\phi)$ is irreducible; see, e.g., \cite{yt}.  One way to see this it as follows.  The group $B$ of upper-triangular matrices acts on $X$, and it preserves the $E_i$, so it also acts on each $D_\lambda(\phi)$.  On the other hand, from the matrix representatives, it is easy to see that $\Omega_\mu^\circ$ is the $B$-orbit of $p_J$ (for $J$ corresponding to $\mu$, as usual).  One checks that the conditions defining $D_\lambda$ are satisfied by $p_J$ if and only if $J\leq I$, and it follows that $D_\lambda$ is the union of the cells $\Omega_\mu^\circ$ for $\mu\supseteq\lambda$.

\begin{remark}
A determinantal formula for $s_\lambda(x|u)$ can be proved directly by algebraic means; see \cite[6.7]{mac}.  Using this, one obtains a new proof of the Kempf-Laksov formula.

A similar proof that the double Schur functions represent Schubert classes is given in \cite[\S5]{mihalcea}.
\end{remark}

\subsection{Positivity}\label{ss:pos}

Since the Schubert classes form a $\Lambda$-basis for $H_T^*X$, we can write
\begin{equation}\label{eq:coeffs}
  [\Omega_\lambda]^T\cdot [\Omega_\mu]^T = \sum_\nu c_{\lambda\mu}^\nu(t) [\Omega_\nu]^T,
\end{equation}
for some polynomials $c_{\lambda\mu}^\nu(t)\in\Lambda=\ZZ[t_1,\ldots,t_n]$, homogeneous of degree $|\lambda|+|\mu|-|\nu|$.  Under the map $H_T^*X\to H^*X$, one recovers the structure constants $c_{\lambda\mu}^\nu=c_{\lambda\mu}^\nu(0)$ for ordinary cohomology.  In fact, the integers $c_{\lambda\mu}^\nu$ are nonnegative---using the Kleiman-Bertini theorem, they count the number of points in a transverse intersection of generic translates of three Schubert varieties.  (There are also many combinatorial formulas for these numbers, which are called the {\em Littlewood-Richardson coefficients}.)

A remarkable fact about the equivariant coefficients $c_{\lambda\mu}^\nu(t)$ is that they also exhibit a certain positivity property:

\begin{theorem}\label{t:graham}
The polynomial $c_{\lambda\mu}^\nu(t)$ lies in $\ZZ_{\geq 0}[t_2-t_1,t_3-t_2,\ldots,t_n-t_{n-1}]$.
\end{theorem}

This is a special case of a theorem of W.~Graham \cite{gr-pos}, who used a degeneration argument.  We will sketch a different proof based on a transversality argument, from \cite{and}.

\begin{proofof}{Sketch of proof}
First we observe that there are equivariant Poincar\'e dual classes, given by the \emph{opposite} Schubert varieties.  Specifically, let $\tilde{E}_i$ be the span of $\{e_n,e_{n-1},\ldots,e_{n+1-i}\}$, and  the sequence
\[
  \tilde{E}_1 \hookrightarrow \tilde{E}_2 \hookrightarrow \cdots \hookrightarrow \tilde{E}_n = \CC^n \xrightarrow{\tilde\phi} Q.
\]
Define $\tilde\Omega_\lambda$ to be the degeneracy locus $D_{\tilde\lambda}(\tilde\phi)$ for this sequence, where $\tilde\lambda$ is the partition whose shape (rotated $180$ degrees) is the complement to $\lambda$ inside the $k\times(n-k)$ rectangle.  Equivalently, let $\tilde{B}$ be the subgroup of lower triangular matrices; then $\tilde\Omega_\lambda$ is the closure of the cell $\tilde\Omega_\lambda^\circ$ obtained as the $\tilde{B}$-orbit of $p_I$.

From the description as a $\tilde{B}$-orbit closure, it is easy to see that the intersection $\Omega_\lambda \cap \tilde\Omega_\lambda$ consists of the single point $p_I$, and is transverse there.  Applying the integration formula, we see that
\[
  \rho_*([\Omega_\lambda]^T\cdot[\tilde\Omega_\mu]^T) = \delta_{\lambda\mu}
\]
in $\Lambda$, and applying this to both sides of \eqref{eq:coeffs},
\[
  c_{\lambda\mu}^\nu(t) = \rho_*([\Omega_\lambda]^T\cdot[\Omega_\mu]^T\cdot[\tilde\Omega_\nu]^T).
\]
The idea of the proof is find a subvariety $Z$ in the approximation space $\EE_m \times^T X$ whose class is equal to $[\Omega_\lambda]^T\cdot[\Omega_\mu]^T\cdot[\tilde\Omega_\nu]^T$, for a special choice of approximation $\EE_m \to \BB_m$.  Pushing forward the class of such a $Z$ yields an effective class in $H^*\BB_m$, which corresponds to $c_{\lambda\mu}^\nu(t)$.

To set this up, fix an isomorphism $T\isom (\CC^*)^n$ using the following basis for the character group:
\[
  t_1-t_2,\,t_2-t_3,\ldots,\,t_{n-1}-t_n,\,t_n.
\]
(The reason for this choice will become clear later.)  Now for $m\gg0$, take $\BB_m=(\PP^m)^n$, and write $M_i = \Oo_i(-1)$ for the pullback via the $i$th projection, as in Example~\ref{ex:Tn}.  Identifying $H^*\BB_m$ with $\Lambda=\ZZ[t_1,\ldots,t_n]$, we have $c_1(M_i) = t_i-t_{i+1}$ for $1\leq i\leq n-1$, and $c_1(\Oo_n(-1))=t_n$.

Note that every effective class in $H^*\BB_m$ is a nonnegative linear combination of monomials in $t_2-t_1,\ldots,t_n-t_{n-1},-t_n$, so if we find $Z$ as above, the positivity theorem will be proved.  (In fact, the whole setup may be pulled back via the projection onto the first $n-1$ factors of $\BB_m=(\PP^m)^n$, so that $t_n$ does not contribute, and the class lies in the claimed subring of $\Lambda$.)

To find $Z$, we construct a group action on the approximation space.  First some notation.  Set $L_i=M_i\otimes M_{i+1}\otimes \cdots \otimes M_n$, so $c_1(L_i) = t_i$.  Using the standard action of $T$ on $\CC^n$, we have $\EE_m\times^T \CC^n = L_1\oplus \cdots \oplus L_n = E$ as vector bundles on $\BB_m$.  Moreover, the flags used to define the Schubert varieties are replaced with $E_i = L_1\oplus \cdots \oplus L_i$ and $\tilde{E}_i = L_n \oplus \cdots \oplus L_{n+1-i}$, and
\[
  \OOmega_\lambda = \EE_m\times^T \Omega_\lambda \quad \text{ and } \tilde\OOmega_\lambda = \EE_m \times^T\tilde\Omega_\lambda
\]
are corresponding degeneracy loci.

The key observation is that the vector bundle $End(E) = \bigoplus_{i,j} L_j^\vee\otimes L_i$ has global sections in ``lower-triangular matrices'', since when $i\geq j$, the line bundle $L_j^\vee\otimes L_i = M_j^\vee \otimes \cdots M_{i-1}^\vee$ is globally generated.  The bundle of invertible endomorphisms $Aut(E) \subset End(E)$ is a group scheme over $\BB_m$, and its group of global sections $\Gamma_0$ maps surjectively onto $\tilde{B}$ by evaluation at each fiber.  Including the action of the group $G=(PGL_{m+1})^{\times n}$ acting transitively on the base $\BB_m$, it follows that the opposite Schubert bundles $\tilde\OOmega_\lambda$ are the orbits for a connected group $\Gamma=G\rtimes \Gamma_0$ acting on $\EE_m\times^T X$.

An application of the Kleiman-Bertini theorem guarantees that there is an element $\gamma\in\Gamma$ such that $Z = \OOmega_\lambda \cap \gamma\cdot\OOmega_\mu \cap \tilde\OOmega_\nu$ is a proper intersection, so it has codimension $|\lambda|+|\mu|-|\nu|+\dim X$ in $\EE_m\times^T X$.  By construction, $[Z]=[\OOmega_\lambda]\cdot[\OOmega_\mu]\cdot[\tilde\OOmega_\nu] = [\Omega_\lambda]^T\cdot[\Omega_\mu]^T\cdot[\tilde\Omega_\nu]^T$, so we are done.
\end{proofof}

\begin{remark}
A similar argument works to establish positivity for any homogeneous space $X=G/P$, so one can recover the general case of Graham's theorem.  An intermediate approach is to observe that the group $\Gamma$ has a unipotent subgroup with finitely many orbits on $\EE_m \times^T X$, so a theorem of Kumar and Nori \cite{kn} guarantees positivity.
\end{remark}

\begin{remark}
From the point of view of degeneracy loci, the coefficients $c_{\lambda\mu}^\nu(t)$ are universal structure constants.  That is, for a fixed degeneracy problem $E_1 \hookrightarrow \cdots \hookrightarrow E_n=E \xrightarrow{\phi} F$ on an arbitrary (Cohen-Macaulay) variety $X$, the subring of $H^*X$ generated by the classes $[D_\lambda(\phi)]$ has structure constants $c_{\lambda\mu}^\nu(u)$, where $u_i=c_1(E_i/E_{i-1})$ is the pullback of $t_i$ by the classifying map $X \to \BB B$.  The positivity theorem therefore implies a corresponding positivity for intersections of degeneracy loci on any variety.
\end{remark}

\subsection{Other directions}

A great deal of recent work in algebraic geometry involves equivariant techniques, either directly or indirectly.  Without pretending to give a complete survey of this work, we conclude by mentioning a small sampling of these further applications.

\bigskip
\noindent
{\bf Generalized Schubert calculus.}  
In the case of the Grassmannian, there are combinatorial formulas for the equivariant coefficients $c_{\lambda\mu}^\nu(t)$, due to Knutson and Tao \cite{KnutsonTao2003} and Molev \cite{molev}.  However, the question remains open for other homogeneous spaces $G/P$, even in the case of ordinary cohomology!  It can happen that the extra structure present in equivariant cohomology simplifies proofs; such is the case in \cite{KnutsonTao2003}.  Some formulas for ordinary cohomology of \emph{cominiscule} $G/P$ were given by Thomas and Yong \cite{ty}.

\bigskip
\noindent
{\bf Degeneracy locus formulas.}  
Extending the Kempf-Laksov formula, there are formulas for {\em (skew-)symmetric degeneracy loci}, due to J\'ozefiak-Lascoux-Pragacz \cite{jlp}, Harris-Tu \cite{ht}, Lascoux-Pragacz \cite{lp}, and others.  These are equivalent to equivariant Giambelli formulas in the Lagrangian and orthogonal Grassmannians, which are homogeneous spaces $Sp_{2n}/P$ and $SO_{2n}/P$.  Working directly in equivariant cohomology, Ikeda et al. have given equivariant Giambelli formulas for these spaces \cite{imn}.

\bigskip
\noindent
{\bf Thom polynomials.}  
The theory of Thom polynomials is one of the origins of equivariant techniques in algebraic geometry; see, e.g., \cite{FeherRimanyi2003}.  These are universal polynomials related to singularities of mappings, and can be interpreted as equivariant classes of certain orbit closures: for an algebraic group $G$ acting linearly on a vector space $V$, one has $[\overline{G\cdot v}]^T$ in $H_T^*(V)\isom\Lambda$.  They are difficult to compute in general, and have been studied recently by Feh\'er-Rim\'anyi \cite{FeherRimanyi2003}, Kazarian \cite{kazarian}, and Pragacz-Weber \cite{pw}, among others.  An important special case is where $G$ is the product of two groups of upper-triangular matrices, acting on the space of $n\times n$ matrices; a detailed combinatorial study of this was carried out by Knutson-Miller \cite{km}.





\noindent
\footnotesize{\textsc{Department of Mathematics, University of Washington, Seattle, WA 98195}}

\end{document}